\documentclass{article}
\usepackage{fullpage,url}
\usepackage{amsmath,amsthm,amssymb}
\usepackage{graphicx}
\newtheorem{theorem}{Theorem}[section]
\newtheorem{lemma}[theorem]{Lemma}
\newtheorem{prop}[theorem]{Proposition}
\newtheorem{cor}[theorem]{Corollary}

\newtheorem{conj}{Conjecture}
\newtheorem{defn}{Definition}

\newcommand{\ba}{\mathbf{a}}
\newcommand{\bb}{\mathbf{b}}
\newcommand{\bc}{\mathbf{c}}
\newcommand{\be}{\mathbf{e}}
\newcommand{\bh}{\mathbf{h}}
\newcommand{\bi}{\mathbf{i}}
\newcommand{\br}{\mathbf{r}}
\newcommand{\bs}{\mathbf{s}}

\newcommand{\bv}{\mathbf{v}}

\newcommand{\bx}{\mathbf{x}}
\newcommand{\by}{\mathbf{y}}
\newcommand{\bz}{\mathbf{z}}

\newcommand{\cH}{\mathcal{H}}

\DeclareMathOperator{\Tr}{Tr}
\DeclareMathOperator{\sgn}{sgn}
\DeclareMathOperator{\spn}{span}

\allowdisplaybreaks

\title{Determinants of Steiner Distance Hypermatrices}
\author{Joshua Cooper, Zhibin Du}
\date{\today}
\begin{document}
\maketitle

\begin{abstract}
    Generalizing work from the 1970s on the determinants of distance hypermatrices of trees, we consider the hyperdeterminants of order-$k$ Steiner distance hypermatrices of trees on $n$ vertices.  We show that they can be nearly diagonalized as $k$-forms, generalizing a result of Graham-Lov\'{a}sz, implying a tensor version of ``conditional negative definiteness'', providing new proofs of previous results of the authors and Tauscheck, and resolving the conjecture that these hyperdeterminants depend only on $k$ and $n$ -- as Graham-Pollak showed for $k=2$.  We conclude with some open questions.
\end{abstract}

\section{Introduction}

Graham and Pollak (\cite{GrPo71}) showed that the determinant of the distance matrix $D(T)$ of a tree $T$ on $n$ vertices only depends on $n$, and provided a formula for it: $(1-n)(-2)^{n-2}$.  Later, Graham and Lov\'{a}sz (\cite{GrLo78}) went  further by exploiting a ``nearly diagonalizing'' congruence of $D(T)$ as a quadratic form, i.e., choosing a matrix $P$ so that $P^{T}D(T)P$ is diagonal except for its first (or last) row and column.  The present authors and Tauscheck (\cite{CoDu24, CoTa24,CoTa24b}) considered  {\em Steiner distance hypermatrices} of order $k \geq 2$, a natural generalization of distance matrices (which correspond to the case of $k=2$), showing that the {\em hyperdeterminant} is $0$ iff (a) $n\geq 3$ and $k$ is odd, (b) $n=1$, or (c) $n=2$ and $k
\equiv 1 \pmod{6}$.  They conjectured that the hyperdeterminant's value -- not just its nonvanishing -- only depends on $n$ and $k$.  Here, we prove this conjecture and provide new proofs of some of the aforementioned results by showing that the Graham-Lov\'{a}sz near-diagonalization fully generalizes to $k > 2$.

Throughout, we consider a tree $T$ with vertex set $[|V(T)|] = \{1,\ldots,n\}$ where $n = |V(T)|$ and edge set $E(T)$.  Let $T$ be a tree and $S(U)$ the Steiner distance of the set $U \subseteq V(T)$, i.e., the fewest number of edges in any connected subgraph of $T$ whose vertex set contains $U$.  We often suppress braces in this notation, i.e., $S(x_1,\ldots,x_N) := S(\{x_1,\ldots,x_N\})$.  Given a vector $\bi \in V(T)^k$, we write $T_\bi$ for the Steiner tree of $\bi = \{i_1,\ldots,i_k\}$, i.e., the (unique) subtree $T' \subseteq T$ so that $|E(T')|=S(i_1,\ldots,i_k)$.  The interested reader is referred to \cite{Ma17} for much more on Steiner distance.  Fixing $k \geq 1$, write $M$ for the order-$k$ Steiner distance hypermatrix of $T$, i.e., for $\bi = (i_1,\ldots,i_k) \in V(T)^k$, define $M_{\bi} = S(\bi)$, and for vectors $\bx_1,\ldots,\bx_k \in \mathbb{R}^{V(T)}$, write 
$$
M(\bx_1,\ldots,\bx_k) := \sum_{\bi \in V(T)^k} M_{\bi} \prod_{j=1}^k x_{ji_j},
$$
where we denote the $j$-th component of a vector represented by any bolded letter $\mathbf{u}$ by the unbolded $u_j$, so that here the $j$-th component of $\bx_i$ is written $x_{ij}$.  Note that $M(\bx_1,\ldots,\bx_k)$ is a multilinear form, and setting $\bx_i = \bc$ for each $i$ yields a $k$-form (i.e., a homogeneous polynomial of degree $k$) in the coordinates of the vector $\bc$.

The {\em hyperdeterminant} of a fully symmetric order-$k$ hypermatrix (equivalently: a $k$-form) is, in analogy to determinants, the unique monic irreducible polynomial in the hypermatrix entries whose vanishing is equivalent to the existence of a nontrivial solution to the system of equations given by $\nabla M(\bx,\ldots,\bx) = \mathbf{0}$, i.e., 
$$
\sum_{\bi \in V(T)^{k-1}} M_{a\bi} \prod_{j=1}^{k-1} x_{i_j} = 0
$$
for each $a \in [n]$.  See \cite{Qi05} for more on the properties of this hyperdeterminant, which parallels the classical theory of determinants and much of which is derived from the theory of resultants via \cite{GeKaZe94}.  

In Section \ref{sec:only_on_H}, we show that the action of $M(\cdot,\ldots,\cdot)$ on the orthogonal complement $\mathcal{H}_n$ of $\mathbb{J}$, the all-ones vector, has a particularly compact form.  This implies that the map vanishes on $\mathcal{H}_n$ if $k \geq 3$ is odd and $n \geq 2$, and it is negative on $\mathcal{H}_n$ if $k \geq 2$ is even and $n \geq 3$.  This latter negativity property for $k=2$ is known as ``conditionally (strictly) negative definite'' in the matrix theory literature, and it has played an important role in explaining the relationship between kernels, metrics, and probability measures, especially in work of Kolmogorov, Schoenberg, and L\"{o}wner; see, for example, \cite{BaRa97, Jo15, Mi86}.  We then use a polarization identity due to Thomas to provide a diagonalization of $M$ on $\mathcal{H}_n$.  In Section \ref{sec:near-diagonalization}, we show that, by taking an appropriate basis of $\mathcal{H}_n$ and a vector other than $\mathbb{J}$ not contained in this codimension-$1$ subspace, it is possible to nearly diagonalize $M$ to a hypermatrix which does not depend on $T$ except via $k$ and $n$.  Since the change-of-basis matrix we employ is the zeta function of $T$ considered as a poset rooted at some vertex, it has determinant $1$, implying that the hyperdeterminant of $M$ also does not depend on $T$.  We conclude with a conjecture about the sign of $\det(M)$ and provide some evidence to support it.

\section{Action on $\mathbb{J}^\perp$} \label{sec:only_on_H}

For an edge $e \in E(T)$, let $A(e)$ and $B(e)$ denote the sets of vertices in the two components of $T-e$ (with ambiguity resolved arbitrarily).  For a tree $T$ on $n$ vertices, let $\ba_e$ be the $n$-vector with $i$-th component $1$ if $i \in A(e)$ and $0$ otherwise; write $\mathbb{J} := (1,\ldots,1) \in \mathbb{R}^n$ for the all-ones vector; and write $\ba'_e = \ba_e-\mathbb{J} |A(e)|/n$.  It is easy to check that $\mathcal{C}(T) = \{\ba'_e\}_{e \in E(T)}$ is a basis for $\cH_n$, the orthogonal complement of the span of the all-ones vector $\mathbb{J}$.

\begin{theorem} \label{thm:mainthm}
    For a tree $T$ on $n \geq 2$ vertices and a vector $\bc \in \mathbb{R}^n$, write $\alpha_e = \bc^T \ba_e$ and $C = \bc^T \mathbb{J}$.  Then
    $$
    M(\overbrace{\bc,\ldots,\bc}^k) =  \sum_{e \in E(T)} \left ( C^k - \alpha_e^k - (C - \alpha_e)^k \right ).
    $$
\end{theorem}
\begin{proof}
First, write $\bb_e = \mathbb{J} - \ba_e$, the indicator vector of the set $B(e)$.  Note that $\beta_e := \bc^T \bb_e = \bc^T \mathbb{J} - \bc^T \ba_e = C - \alpha_e$.
Let $\bi = (i_1,\ldots,i_{k}) \in V(T)^k$. Write $c^\bi$ to denote $\prod_{j=1}^k c_{i_j}$ if $\bc = (c_1,\ldots,c_n) \in \mathbb{R}^n$. Then
\begin{align*}
M(\overbrace{\bc,\ldots,\bc}^k) &= \sum_{\bi \in V(T)^k} S(\bi) c^\bi  \\
&= \sum_{\bi \in V(T)^k} \sum_{e \in E(T_\bi)} c^\bi.
\end{align*}
Consider the inner sum above; its terms are functions of a pair $(\bi,e)$, where $\bi=(i_1,\ldots,i_k) \in V(T)^k$ and $e \in E(T_\bi)$.  We may decompose this set of pairs according to the coordinates of $\bi$ which belong to $A(e)$ versus $B(e)$.  That is, for each $\mathbf{u},\mathbf{v} \in V(T)^\ast$, let $E(\mathbf{u},\mathbf{v})$ denote the set of edges $e$ of $T$ for which $u \in A(e)$ for all $u$ a coordinate of $\mathbf{u}$ and $v \in B(e)$ for all $v$ a coordinate of $\mathbf{v}$.  If $J = \{j_1 < \cdots < j_t\} \subseteq [k]$, write $\bi_J$ for $(i_{j_1},\ldots,i_{j_t})$.  Note that $e \in E(T_\bi)$ iff $e \in E(\bi_J,\bi_{[k] \setminus J})$ for some $J$ with $1 \leq |J| \leq k-1$.  Then we may write
\begin{align*}
\sum_{\bi \in V(T)^k} \sum_{e \in E(T_\bi)} c^\bi &= \sum_{\bi \in V(T)^k} \sum_{\substack{J \subseteq [k] \\ 1 \leq |J| \leq k-1}} \sum_{e \in E(\bi_J,\bi_{[k] \setminus J})} c^{\bi} \\ 
&= \sum_{\substack{J \subseteq [k] \\ 1 \leq |J| \leq k-1}} \sum_{\bi \in V(T)^k} \sum_{e \in E(\bi_J,\bi_{[k] \setminus J})}  c^\bi \\
&= \sum_{t=1}^{k-1} \sum_{J \in \binom{[k]}{t}} \sum_{e \in E(T)} \sum_{\substack{\bi \in V(T)^{k} \\ \text{s.t. } \bi_J \in A(e)^t \\ \text{\& } \bi_{[k] \setminus J} \in B(e)^{k-t}}} c^\bi \\ 
&= \sum_{t=1}^{k-1} \sum_{J \in \binom{[k]}{t}} \sum_{e \in E(T)} \sum_{\mathbf{u} \in A(e)^{t}} \sum_{\mathbf{v} \in B(e)^{k-t}} c^\mathbf{u} c^\mathbf{v} \\ 
&= \sum_{e \in E(T)}  \sum_{t=1}^{k-1} \binom{k}{t} \sum_{\mathbf{u} \in A(e)^{t}} \sum_{\mathbf{v} \in B(e)^{k-t}} c^\mathbf{u} c^\mathbf{v}  \\ 
&= \sum_{e \in E(T)}  \sum_{t=1}^{k-1} \binom{k}{t} \left ( \sum_{u \in A(e)} c_u \right )^t \left ( \sum_{v \in B(e)} c_v \right )^{k-t}  \\ 
&= \sum_{e \in E(T)} \sum_{t=1}^{k-1} \binom{k}{t} \alpha_e^t \beta_e^{k-t} \\
&=  \sum_{e \in E(T)} \left ( (\alpha_e + \beta_e)^k - \alpha_e^k - \beta_e^k \right ),
\end{align*}
from which the desired conclusion follows.
\end{proof}

We now deduce a few corollaries of Theorem \ref{thm:mainthm}.  The next result, which follows immediately from the preceding proof, states that  $M(\bc,\ldots,\bc)$ as a polynomial in $\{c_i\}_i$ acting on $\cH_n$ can be expressed in an exceptionally simple manner.

\begin{cor} \label{cor:diagonalform}
    For a tree $T$ on $n\geq 2$ vertices, $k \geq 2$ even, and a vector $\bc \in \cH_n$, if $\alpha'_e = \bc^T \ba'_e$, then
    $$
    M(\overbrace{\bc,\ldots,\bc}^k) = -2 \sum_{e \in E(T)} \alpha_e^{\prime k},
    $$
    so the left-hand side is a diagonal form of Waring rank at most $k-1$.
\end{cor}
\begin{proof}
    Note that $C = 0$ and $\beta_e = -\alpha_e$.  Furthermore, $\bc^T \ba_e = \bc^T \ba'_e$ for any $\bc \in \cH_n$. Since $k$ is even, then, applying Theorem \ref{thm:mainthm} yields
    $$
    M(\overbrace{\bc,\ldots,\bc}^k) = -\sum_{e \in E(T)} \alpha_e^k - \sum_{e \in E(T)} \beta_e^k = - 2 \sum_{e \in E(T)} \alpha_e^k.
    $$
\end{proof}

\begin{cor} For a tree $T$ on $n\geq 2$ vertices, $k \geq 2$ odd, and a vector $\bc \in \cH_n$, then
$$
    M(\overbrace{\bc,\ldots,\bc}^k) =  0.
$$
\end{cor}
\begin{proof} 
    Note that $C = 0$ and $\beta_e = -\alpha_e$, so
    \begin{align*}
    M(\overbrace{\bc,\ldots,\bc}^k) &=  \sum_{e \in E(T)} \left ( C^k - \alpha_e^k - (C - \alpha_e)^k \right ) \\
    &=  \sum_{e \in E(T)} \left ( - \alpha_e^k - (- \alpha_e)^k \right ) = 0.
    \end{align*}
\end{proof}

\begin{defn}
    An order-$k$, dimension-$n$ hypermatrix $A$ is called {\em conditionally negative definite} if, for all nonzero $\bc \in \mathbb{R}^n$ so that $\bc^T \mathbb{J} = 0$,
$$
A(\overbrace{\bc,\ldots,\bc}^k) \leq 0.
$$
We say that $A$ is {\em conditionally strictly negative definite}  if the above inequality is strict. 
\end{defn}

\begin{theorem} \label{thm:CND} Let $T$ be a tree on $n \geq 3$ vertices, and suppose $k \geq 2$ is even.  Then $M$ is conditionally strictly negative definite.
\end{theorem}
\begin{proof}
We apply the preceding corollary:
    $$
    M(\overbrace{\bc,\ldots,\bc}^k) = - 2 \sum_{e \in E(T)} \alpha_e^k.
    $$
The result then follows from the facts that $k$ is even and there must be an edge so that $\sum_{i \in A(e)} c_i \neq 0$.  Indeed, let $u$ be any coordinate so that $c_u \neq 0$, and let $e_1,\ldots,e_{\deg(u)}$ be all edges incident to $u$, where $\deg(u)$ denotes the degree of $u$.  Assume each $A(e)$ is chosen so that $u \not \in A(e)$.  Then every vertex is either $u$ itself, or belongs to exactly one $A(e_j)$, $j = 1, \ldots, \deg(u)$.  Thus,
$$
\sum_{i =1}^n c_i = c_u + \sum_{j=1}^{\deg(u)} \sum_{i \in A(e_j)} c_i = 0,
$$
so at least one of the sums $\sum_{i \in A(e_j)} c_i$ is nonzero.
\end{proof}

For a $k$-linear map $u : E^k \rightarrow F$, write $\tilde{u} : E \rightarrow F$ for the map so that, for each $\bx \in E$,
$$
\tilde{u}(\bx) = u(\bx,\ldots,\bx).
$$
For $\bh \in E$ and a map $v : E \rightarrow F$, write $\Delta_\bh v : E \rightarrow F$ for the map so that, for each $\bx \in E$,
$$
(\Delta_\bh v)(\bx) = v(\bx+\bh)-v(\bx).
$$
Finally, write $\Tr$ for the map which sends $v$ to $v(0)$.

\begin{theorem}[\cite{Th11}]\label{thm:polarization} Let $E$ and $F$ be linear spaces over a field $K$ of characteristic zero, and let
$u : E^k \rightarrow F$ be a symmetric $k$-linear map. Then we have the polarization identity
$$
u(\bx_1, \ldots , \bx_k) = \frac{1}{k!} \Tr \Delta _{\bx_k} \Delta_{\bx_{k-1}} \ldots \Delta_{\bx_1} \tilde{u}.
$$    
\end{theorem}

Write $P_T$ for the matrix giving components of vectors in $\mathbb{R}^n$ in the basis $\mathcal{C}(T)$, i.e., the $(n-1) \times n$ matrix with rows indexed by $E(T)$ whose row $e$ is the vector $\ba_e^{\prime T}$.  We use the symbol $\mathbb{I}_n^k$ for the order-$k$, dimension-$n$ ``identity hypermatrix'', i.e., the hypermatrix whose $(i_1,\ldots,i_k) \in [n]^k$ entry is $0$ unless $i_1 = \cdots = i_k$, in which case the entry is $1$.

Write $P'_T$ for the matrix whose first $n-1$ rows are $P_T$ and whose $n$-th row is $\mathbb{J}^T$.  We consider the columns of $P^{\prime-1}_T$ to be indexed by $E(T)$ in concordance with the edge ordering on $E(T)$ which gives the row labels of $P_T$, except for its final column.

\begin{lemma}
    For each $T$, the column of $P^{\prime-1}_T$ indexed by $e=\{i,j\}$ has $-1$ in row $i$, $1$ in row $j$, and $0$ in every other row, where $j \in A(e)$ and $i \in B(e)$; and the final column of $P^{\prime-1}_T$ consists of $\mathbb{J}/n$.  In other words, the first $n-1$ columns of $P^{\prime-1}_T$ are the incidence matrix of $T$ if all of its edges $e$ are oriented towards $A(e)$, and its final column is $\mathbb{J}/n$.
\end{lemma}
\begin{proof}
    Let $X$ be the matrix described by the statement.  We show that $P'_T X = I_n$.  Let $e \in E(T)$ be the index of row $\br_e$ of $P'_T$, $f=ij$ in $E(T)$ the index of a column $\bs$ of $X$, where $j \in A(e)$ and $i \in B(e)$.  Then $\br_e \bs$ is the $(e,f)$ entry of $P'_T X$ and equals $-r_{ei}+r_{ej}$.  First, suppose $e \neq f$.  Note that $i$ and $j$ are either both in $A(e)$, or both in $B(e)$.  Thus, $r_{ei} = r_{ej}$ and $-r_{ei}+r_{ej}=0$.  On the other hand, if $e = f$, then $-r_{ei}+r_{ej}=-(-|A(e)|/n)+(1-|A(e)|/n)=1$.  The final row $\mathbb{J}^T$ of $P'_T$ has dot product $0$ with each column of $X$ except the last one, with which it has dot product $1$.  So $P'_T X = I_n$, and the conclusion follows. 
\end{proof}

To interpret the preceding result: The columns of the incidence matrix of $T$, oriented towards the sets $\{A(e)\}_e$ are a basis $\mathcal{I}_T$ of $\cH_n$.  Together with the vector $\mathbb{J}/n$, they are the columns of $P^{\prime-1}_T$ and form a basis of $\mathbb{R}^n$.  The basis change matrix from $\mathcal{I}_T \cup \{\mathbb{J}/n\}$ to the elementary basis $\mathcal{E}$ is $P^{\prime-1}_T$, and so the basis change matrix from $\mathcal{E}$ to $\mathcal{I}_T \cup \{\mathbb{J}\}$ is $P'_T$.   Thus, if we consider a vector $\bc \in \cH_n$ expressed with respect to $\mathcal{E}$, then multiplication by $P_T$ changes its representation to $\mathcal{I}_T$, because the component in the direction of $\mathbb{J}$ is $0$ and is discarded by using $P_T$ instead of $P'_T$.

\begin{cor} \label{cor:diagonalform2} For even $k\geq 2$ and any $\bc_1,\ldots,\bc_k \in \cH_n$,
$$
M(\bc_1,\ldots,\bc_k) = (-2 \mathbb{I}_{n-1}^k)(P_T \bc_1,\ldots,P_T \bc_k),
$$
i.e., the action of $M$ on $\cH_n$ is diagonalizable and is represented by $-2\mathbb{I}_{n-1}^k$ in the basis $\mathcal{I}_T$.
\end{cor}
\begin{proof}
    Theorem \ref{thm:polarization} implies that a symmetric $k$-linear form $u : E^k \rightarrow F$ is uniquely determined by its values $\tilde{u}(\bc)$ over all $\bc \in E$.  In this case, we take $E = \cH_n$ and $F = \mathbb{R}$.  Since 
    $$
    \tilde{u}(\bc) = (-2 \mathbb{I}_{n-1}^k)(P_T \bc,\ldots,P_T \bc),
    $$
    for all $\bc \in \cH_n$ by Corollary \ref{cor:diagonalform}, the result follows. 
\end{proof}

Note that the above gives another way to see that $M$ is conditionally negative definite: For $\bc \in \cH_n$,
$$
M(\bc,\ldots,\bc) = -2 \|\bc'\|_k^k
$$
where $\bc'$ is $\bc$ written in the $\mathcal{I}_n$ basis.  This quantity is negative for any nonzero $\bc$.  In addition, for any $\bc \not \in \cH_n$, if we normalize so that $\sum_i c_i = 1$, then 
$$
M(\bc,\ldots,\bc) \geq 0
$$
iff $\|P_T \bc\|_k^k + \|\mathbb{J}-P_T \bc\|_k^k \leq n-1$, by Theorem \ref{thm:mainthm}; for $k$ even, this is a compact set, and the two-sided cone consisting of all its multiples is the largest set on which $M$ is positive semi-definite.

\section{Near-diagonalization} \label{sec:near-diagonalization}

Let $Z$ be the matrix of the zeta function of $T$, considered as a poset $(T,\preceq_v)$ rooted at an arbitrary vertex $v \in T$, i.e., $Z_{xy} = 1$ if $x$ lies on the unique $y-v$ path in $T$, and $0$ otherwise.  Then $Z^{-1}$ is the M\"obius function of $(T,\preceq_v)$, which it is easy to see has $(x,y)$ entry $1$ if $x=y$, $-1$ if $y$ is a neighbor of $x$ with $y \preceq_v x$, and $0$ otherwise.  See \cite{Go18} for details.  Note that, if we choose $v=n$, then $-Z^{-1}$ and $P_T^{\prime -1}$ agree in all entries other than their last column if we choose $A(e)$ to be the set containing $v$ for each $e$.  The $n$-th column of $-Z^{-1}$ is the elementary vector $\be_n$, whereas the $n$-th column of $P_T^{\prime -1}$ is $\mathbb{J}/n$.  If $\mathbf{c} \in \cH_n$, then the last coordinate of $P'_T\bc$ is $0$, so $-Z^{-1}P'_T\bc = P_T^{\prime-1}P'_T \bc = \bc$, i.e., $P'_T \bc = -Z \bc$.  Thus, by eliminating a factor of $(-1)^k=1$, we may restate Corollary \ref{cor:diagonalform2}: under the same hypotheses, 
$$
M(\bc_1,\ldots,\bc_k) = (-2 \mathbb{I}_{n-1}^k)(Z \bc_1,\ldots,Z \bc_k),
$$
However, it turns out that writing $M$ in this basis now makes it possible to write the $k$-linear form in a particularly simple way across all of $\mathbb{R}^n$.  The following is a generalization of a result of Graham-Lov\'{a}sz \cite{GrLo78} (their ``Fact 1''), which is the case of $k=2$.

\begin{theorem} \label{thm:genGL} For any $\bc_1,\ldots,\bc_n$, and $k \geq 2$ even,
$$
M(\bc_1,\ldots,\bc_n) = (U-2(\mathbb{I}_{n-1}^k\oplus [0]) )(Z\bc_1,\ldots,Z\bc_n),
$$
where $U$ is the hypermatrix whose $(i_1,\ldots,i_k)$ entry is $(-1)^{t-1}$ if there exists a vertex $w \in [n-1]$ and a set $S \subseteq [k]$ with $|S|=t \in [1,k-1]$ so that $i_j = n$ for $j \in [k] \setminus S$ and $i_j=w$ for $j \in S$, and $0$ otherwise.  If $k$ is odd, then
$$
M(\bc_1,\ldots,\bc_n)= U(Z\bc_1,\ldots,Z\bc_n) .
$$
\end{theorem}

\begin{proof}
By Theorem \ref{thm:polarization}, it suffices to verify the result for the corresponding $k$-form. To that end, let $\bc' = Z\bc$, and write $c'_j$ for the $j$-th component of $\bc'$.  Then, for $k$ even,
\begin{align}
    (U-2(\mathbb{I}_{n-1}^k\oplus [0]) )(Z\bc,\ldots,Z\bc) &= (U-2 (\mathbb{I}_{n-1}^k\oplus [0]) )(\bc',\ldots,\bc') \label{eq4} \\
    &= \left ( \sum_{t=1}^{k-1} \binom{k}{t} (-1)^{t-1} \sum_{w=1}^{n-1} c^{\prime t}_w c^{\prime k-t}_n \right ) - 2  \sum_{w=1}^{n-1} c_w^{\prime k} \nonumber \\   
    &= \sum_{w=1}^{n-1} \left (-(c'_w-c'_n)^k + (-c'_n)^{k} + c_w^{\prime k} \right ) - 2 \sum_{w=1}^{n-1} c_w^{\prime k} \nonumber \\
    &= \sum_{w=1}^{n-1} \left (c_n^{\prime k} - c_w^{\prime k} -(c'_n-c'_w)^k \right ) . \nonumber
\end{align}
Label the edges of $T$ as $\{e_1,\ldots,e_{n-1}\}$, where $e_j$ is the edge incident to vertex $j$ whose other vertex is the neighbor of $j$ closer to $n$.  Write $\alpha_j = \alpha_{e_j}$ and $\alpha'_j = \alpha'_{e_j}$ for succinctness. Now, $P'_T \bc$ is a vector whose $j$-th coordinate is $\ba_{e_j}^{\prime T} \bc = \alpha'_j$ for $j < n$ and whose $n$-th coordinate is $\mathbb{J}^T \bc$.  

From the discussion preceding the statement of the present theorem, it is clear that the first $(n-1)\times(n-1)$ principal submatrix of $-ZP_T^{\prime -1}$ is the identity matrix.  It remains to describe its last column, which is $-Z \mathbb{J}/n$.  The $j$-th coordinate of $Z \mathbb{J}$ is the sum of all entries in the $j$-th row of $Z$, i.e., the number of vertices $v$ in $T$ so that $j \preceq_n v$.  Therefore, $-ZP_T^{\prime -1}$ is the matrix whose first $(n-1)\times(n-1)$ principal submatrix is the identity, and whose last column has entry $-|\{v \in T : j \preceq_n v\}|/n$ in row $j$, which equals $-|B(e_j)|/n$ if $j < n$ or $-1$ if $j=n$.  
Thus, $\bc' = Z \bc = (-Z P_T^{\prime -1}) (-P'_T \bc)$ is a vector whose $j$-th coordinate is $|B(e_j)|C/n-\alpha'_j$ for $j < n$, or $C$ for $j=n$, where recall that $C = \mathbb{J}^T \bc$.  Note that
$$
\frac{|B(e_j)|C}{n} - \alpha'_j = \frac{(n-|A(e_j)|)C}{n} - \alpha'_j = C - \alpha'_j - \frac{|A(e_j)|C}{n} = C - \alpha_j .
$$
That is,
$$
c'_j = \left \{ 
\begin{array}{ll}
C - \alpha_j & \textrm{if $j < n$} \\
C & \textrm{if $j=n$.}
\end{array}
\right .
$$
Therefore, we may rewrite (\ref{eq4}) as
\begin{align*}
    (U-2(\mathbb{I}_{n-1}^k\oplus [0]) )(Z\bc,\ldots,Z\bc) &= \sum_{j=1}^{n-1} \left ( C^{k} - (C-\alpha_j)^{k} -\alpha_j^k \right ) ,
\end{align*}
which agrees with the expression for $M(\bc,\ldots,\bc)$ in Theorem \ref{thm:mainthm}.  For any $k$, then 
\begin{align*}
    U(Z\bc,\ldots,Z\bc) &= U(\bc',\ldots,\bc') \\
    &= \sum_{t=1}^{k-1} \binom{k}{t} (-1)^{t-1} \sum_{j=1}^{n-1} c^{\prime t}_j c^{\prime k-t}_n  \\   
    &= - \sum_{j=1}^{n-1} \sum_{t=1}^{k-1} \binom{k}{t} (-c^{\prime}_j)^t c^{\prime k-t}_n  \\   
    &= - \sum_{j=1}^{n-1} \left ((c'_n-c'_j)^k - c^{\prime k}_n - (-c'_j)^{k} \right )   \\
    &= \sum_{j=1}^{n-1} \left (C^k + (\alpha_j - C)^{ k} -\alpha_j^k \right ), 
\end{align*}
which agrees with the expression for $M(\bc,\ldots,\bc)$ in Theorem \ref{thm:mainthm} when $k$ is odd.
\end{proof}

\begin{cor} For $k \geq 2$ even,
$$
    \det(M) = \det(U-2(\mathbb{I}_{n-1}^k\oplus [0])) ,
$$
where $U$ is the hypermatrix whose $(i_1,\ldots,i_k)$ entry is $(-1)^{t-1}$ if there exists a vertex $w \in [n-1]$ and a set $S \subseteq [k]$ with $|S|=t \in [1,k-1]$ so that $i_j = n$ for $j \in [k] \setminus S$ and $i_j=w$ for $j \in S$. In particular, $\det(M)$ depends on $T$ only through $n$ and $k$.
\end{cor}
\begin{proof}
    It follows immediately from Theorem \ref{thm:genGL} and Proposition 4.4 in \cite{HuHuLiQi13} that
    $$
    \det(M) = \det(U-2(\mathbb{I}_{n-1}^k\oplus [0])) \cdot \det(Z)^{(k-1)^n}.
    $$
    However, $Z$ is upper-triangular with all diagonal entries $1$, so $\det(Z)=1$.
\end{proof}

We illustrate the above result by presenting a schematic of the hypermatrix $U-2(\mathbb{I}_{4}^4\oplus [0])$ (i.e., $n=5$ and $k=4$) in Figure \ref{fig1}.

\begin{figure}[h!]
 \centering
  \includegraphics[width=0.5\textwidth]{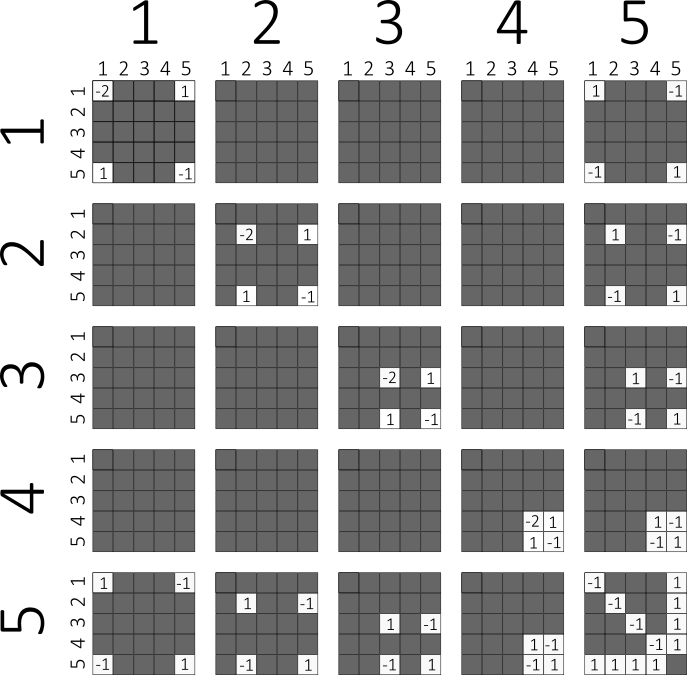}
 \caption{\label{fig1}Depiction of the nearly diagonalized order-$4$ Steiner distance hypermatrix of any tree on $5$ vertices.  Indices of the hypermatrix are represented by small and large numbers at the top and left of the diagram; gray represents a $0$ entry; nonzero entries are shown with their value ($-2$ or $\pm 1$).}
 \end{figure}

\section{Conclusion}

The following conjecture appears in $\cite{CoDu24}$, and has been checked numerically for $(k, n) = (4, 4)$, $(4, 5)$, and $(6, 4)$, it is trivially true for $n \leq 3$ and $k$ odd, and it follows from Graham-Pollak when $k=2$.  Write $\sgn(x)$ to mean $1$ if $x > 0$, $-1$ if $x < 0$, or $0$ if $x=0$.

\begin{conj} \label{conj:sign}
    For any tree on $n$ vertices, $\sgn(\det(M)) = (-1)^{n-1}$ when this quantity is nonzero.
\end{conj}

Here we also verify the conjecture for $n=2$ and all $k$.

\begin{prop} Conjecture \ref{conj:sign} is true for $K_2$ and any $k \geq 2$.
\end{prop}
\begin{proof}
By \cite{CoDu24}, the Steiner order-$k$ hypermatrix eigenvalues of $K_2$ are $-1$ with multiplicity $k-1$, and $\lambda_j = [1+\cos(2\pi j/(k-1))+\mathrm{i} \sin(2 \pi j/(k-1))]^{k-1}-1$ for $j = 0, 1, \dots, k - 2$.  For $j = 0$, we have $\lambda_j = 2^{k-1} - 1 > 0$. It is straightforward to see that $\Re(\lambda_j)=\Re(\lambda_{k-1-j})$ for $j=0,\ldots,k-2$.  If $j = k-1-j$, then $j=(k-1)/2$, which implies that 
$$
\lambda_j = \left [1+\cos \pi + \mathrm{i} \sin \pi  \right ]^{k-1}-1 = (1-1+0)^{k-1}-1 = -1 < 0.
$$
Thus, the positive eigenvalues consist of an even number of the $\lambda_j$ along with $j=0$, i.e., there are an odd number of positive eigenvalues, an odd number of negative eigenvalues, and $\sgn(\det(M)) = -1$.
\end{proof}

We now reprove the conjecture for $k=2$ in order to suggest a route to generalization.

\begin{lemma} \label{lem:A_matrix}
    For each $\bv \in \mathbb{R}^n$, let $\bv_0$ be the projection of $\bv$ onto $\cH_n$ and let $\bv_1 = \bv - \bv_0$ be its projection onto $\spn(\mathbb{J})$.  The linear transformation that maps $\bv = \bv_0 + \bv_1$ to $-\bv_0 + \bv_1$ is represented by the matrix $A = 2J_n/n - I_n$, where $J_n$ is the $n \times n$ all-ones matrix and $I_n$ is the $n \times n$ identity matrix.  Furthermore, $A^{-1} = A$ and $A$ is orthogonal.
\end{lemma}
\begin{proof}
    Let $A = \frac{2}{n} J_n - I_n$.  Then
    \begin{align*}
    A \mathbb{J} &= \frac{2}{n} J_n \mathbb{J} - \mathbb{J} \\
    &= \frac{2}{n} n \mathbb{J} - \mathbb{J} \\
    &= \mathbb{J}.
    \end{align*}
    Next, we show that each element of a spanning set of $\cH_n$ is mapped to its negative.  For $i \neq j$, let $\be_{ij}$ denote the vector in $\mathbb{R}^n$ whose $i$ coordinate is $1$, whose $j$ coordinate is $-1$, and whose other coordinates are $0$.  It is straightforward to see that that $\spn(\{e_{ij}\}_{i \neq j}) = \mathcal{H}_n$ and also that $A \be_{ij} = -\be_{ij}$.  The rest of the claim follows from the fact that $A^2=I_n$ and $A$ is symmetric.
\end{proof}

Set 
$$
C:= M \left( \frac{2}{n}J_n - I_n\right)
$$
and 
$$
D:= M \left( \frac{2}{n}J_n - I_n\right) + \left( \frac{2}{n}J_n - I_n\right) M \, ,
$$
where $M$ is the distance matrix of an $n$-vertex tree.  Clearly,  $D$ is symmetric, though $C$ may not be.

\begin{lemma} \label{lem-2}
$D$ is positive definite.
\end{lemma}
\begin{proof}
    Let $\bx$ be nonzero.  Write $\bx = \by + \bz$ where $\by \in \cH_n$ and $\bz \in \spn(\mathbb{J})$; and let $A = \frac{2}{n} J_n - I_n$.  Note that $A\bz = \bz$ and $A\by = -\by$ by Lemma \ref{lem:A_matrix}.  Thus,
    \begin{align*}
        \bx^T D \bx &= (\by^T + \bz^T) (MA+AM) (\by + \bz) \\
        &= 2\by^T MA \by + 2 \bz^T MA \bz + 2 \by^T(MA + AM) \bz
    \end{align*}
    using the fact that $A$ and $M$ are symmetric. Continuing the calculation,
    \begin{align*}
        \bx^T D \bx &= 2\by^T M A \by + 2 \bz^T MA \bz + 2 \by^T MA \bz + 2 \by^T AM \bz \\
        &= -2\by^T M \by + 2 \bz^T M \bz + 2 \by^T M \bz - 2 \by^T M \bz \\
        &= -2\by^T M \by + 2 \bz^T M \bz > 0,
    \end{align*}
    since $M$ is positive on $\spn(\mathbb{J})$ and negative on $\cH_n$ by Theorem \ref{thm:CND}.
\end{proof}

\begin{lemma}
The matrix $C$ is positive definite.
\end{lemma}

\begin{proof}
Clearly, $D = C + C^T$. It is known that $D$ is positive definite in Lemma \ref{lem-2}. 
That is to say, for any nonzero ${\bf x} \in \mathbb{R}^n$, we have 
$$
{\bf x}^T D {\bf x} > 0 \, .
$$
So
$$
{\bf x}^T (C + C^T) {\bf x} = {\bf x}^T C {\bf x} + {\bf x}^T C^T {\bf x} = 2 {\bf x}^T C {\bf x} > 0 \, ,
$$
as desired.
\end{proof}

Since $C$ is positive definite and $\det(2J_n/n-I_n) = (-1)^{n-1}$, 
$$
\sgn(\det(M)) = \frac{\sgn(\det(C))}{\sgn(\det(2J_n/n - I_n))} = (-1)^{n-1}.
$$

Unfortunately, the above strategy does not easily generalize to $k > 2$, although we obtain a weaker result as follows.  Let $A$ be the matrix mentioned above representing reflection through $\mathbb{J}$, i.e., so that $A \bx = -\bx$ if $\bx \in \mathcal{H}_n$ and $A \bx = \bx$ if $\bx \in \spn(\mathbb{J})$.  Then, for $k=4$, write 
$$
C_1 = M(\cdot,\cdot,\cdot,A\cdot)
$$
and
$$
C_3 = M(\cdot,A\cdot,A\cdot,A\cdot).
$$

\begin{prop} The order-$4$ hypermatrix $C_1 + C_3$ is positive definite.    
\end{prop}
\begin{proof}
Suppose $\bx \in \spn(\mathbb{J})$ and $\by \in \mathcal{H}_n$. Since $M$ is symmetric, $A\bx = \bx$ and $A \by = -\by$, by multilinearity we have
\begin{align*}
C_1(\bx+\by,\bx+\by,\bx+\by,\bx+\by) &= M(\bx+\by,\bx+\by,\bx+\by,\bx-\by) \\
&= M(\bx,\bx,\bx,\bx) + (3-1) M(\bx,\bx,\bx,\by) \\
& \qquad + (3-3) M(\bx,\bx,\by,\by) + (1-3) M(\bx,\by,\by,\by) - M(\by,\by,\by,\by) \\
&= M(\bx,\bx,\bx,\bx) + 2 M(\bx,\bx,\bx,\by) - 2 M(\bx,\by,\by,\by) - M(\by,\by,\by,\by)
\end{align*}
and
\begin{align*}
C_3(\bx+\by,\bx+\by,\bx+\by,\bx+\by) &= M(\bx+\by,\bx-\by,\bx-\by,\bx-\by) \\
&= M(\bx,\bx,\bx,\bx) + (1-3) M(\bx,\bx,\bx,\by) \\
& \qquad + (3-3) M(\bx,\bx,\by,\by) + (3-1) M(\bx,\by,\by,\by) - M(\by,\by,\by,\by) \\
&= M(\bx,\bx,\bx,\bx) - 2 M(\bx,\bx,\bx,\by) + 2 M(\bx,\by,\by,\by) - M(\by,\by,\by,\by).
\end{align*}
Thus,
\begin{equation} \label{eq3}
(C_1+C_3)(\bx+\by,\bx+\by,\bx+\by,\bx+\by) = 2 M(\bx,\bx,\bx,\bx) - 2M(\by,\by,\by,\by) > 0 
\end{equation}
by Theorem \ref{thm:CND} and the fact that $M$ is nonnegative.
\end{proof}

In addition to the above arguments, we have computational evidence that $C_1$ is positive definite (which also implies that $C_3$ is positive definite and Conjecture \ref{conj:sign} for $k=4$).  By \cite{Qi05}, an even-order hypermatrix is positive definite iff all of its ``$H$-eigenvalues'' -- real eigenvalues associated with a real eigenvector -- are positive.  Thus, we ask: is it true that the $H$-eigenvalues of $C_1$ are all positive?  Does this generalize to all even $k$?

\end{document}